\documentclass[a4paper,11pt]{article}

\usepackage{amsfonts}
\usepackage{url}
\usepackage{amssymb}
\usepackage{amsthm}
\usepackage{latexsym,amsmath,amssymb,mathtools}

\usepackage[all]{xy}

\newtheorem{proposition}{Proposition}
\newtheorem{definition}{Definition}
\newtheorem{corollary}{Corollary}

\newtheoremstyle{obs}
{3pt}
{3pt}
{}
{}
{\bfseries}
{.}
{.5em}
{}
\theoremstyle{obs}

\newtheorem{example}{Example}

\newcommand{\hooklongrightarrow}{\lhook\joinrel\longrightarrow}

\def\tabaddress#1{{\small\it\begin{tabular}[t]{c}#1
			\\[1.2ex]\end{tabular}}}

\title{Presymplectic integrators for optimal control problems via retraction maps}
\author{{\sc M. Barbero-Li\~n\'an}
	\thanks{{\bf e}-{\it mail}: m.barbero@upm.es}
	\\
	\tabaddress{Departamento de Matem\'atica Aplicada, Universidad Polit\'ecnica de Madrid, \\ Av. Juan de Herrera 4, 28040 Madrid, Spain}\\
   {\sc D.\ Mart\'{\i}n de Diego\thanks{{\bf e}-{\it mail}:
			david.martin@icmat.es}} \\
		\tabaddress{	Instituto de Ciencias Matem\'aticas (CSIC-UAM-UC3M-UCM)\\ C/Nicol\'as Cabrera 13-15, 28049 Madrid, Spain}
}

\date{\today}

\begin{document}
	\maketitle

\begin{abstract} 
Retractions maps are used to define a discretization of the tangent bundle of the configuration manifold as two copies of the configuration manifold where the dynamics take place. Such discretization maps can be conveniently lifted to the cotangent bundle so that symplectic integrators are constructed for Hamilton's equations. Optimal control problems are provided with a Hamiltonian framework by Pontryagin's Maximum Principle. That is why we use discretization maps and the integrability algorithm to obtain presymplectic integrators for optimal control problems. 

\vspace{4mm}

\end{abstract}

\section{Introduction}

Retraction maps first appear in the literature as a topological construction in 1931~\cite{1931Borsuk}. They did not become a useful tool for designing optimization algorithms on matrix manifolds until the beginning of the XXIst century, see for instance~\cite{AbMaSeBookRetraction}. Recently, we have developed in~\cite{21MBLDMdD} a new notion of discretization map arisen from retraction maps that focuses on discretizing the configuration manifold, instead of the equations of motion. That provides a new approach to discrete mechanics that started with the foundational work~\cite{MW_Acta} based on discretizing the variational principles that define the equations of motion. 

When Lagrangian or Hamiltonian systems are considered, the above-mentioned discretization map is built exploiting the properties of the tangent and cotangent bundle structures~\cite{TuHamilton,YaIs73}. In the Hamiltonian framework we obtain a systematic procedure to construct symplectic numerical methods~\cite{21MBLDMdD}, as in~\cite{hairer}.

In this paper we aim at constructing geometric integrators for optimal control problems. In 1958, Pontryagin's Maximum Principle~\cite{Pontryagin} provided necessary conditions for optimal solutions. Only a few years later, the discrete version of such conditions to obtain numerical methods for optimal control problems appeared~\cite{64Jordan}. Decades later, geometric integrators for optimal control problems have been studied and characterized~\cite{2009ChybaSymplIntOpControl,2011SinaJungeMa,2013FerKoDMdD}, even for singular optimal control problems~\cite{2009DelgadoIbort}. 

Pontryagin's Maximum Principle~\cite{Pontryagin} provides the optimal control problem with Hamiltonian framework~\cite{1997Jurdjevic}.  The results from symplectic geometry~\cite{90GuiStern,LiMarle} together with the integrability algorithm~\cite{1995MMT} can be used to identify the final submanifold where the solutions live so that the discretization maps can be used. A symplectic integrator can be defined as a Lagrangian submanifold~\cite{Weinstein}. However, the geometric integrators for optimal control probelms in this paper will not preserve the symplectic 2-form, but a presymplectic 2-form~\cite{GuzMarr}.

The paper is organized as follows. After recalling Hamilton's equations in Section~\ref{Sec:HamEq}, we summarize how the discretization maps in~\cite{21MBLDMdD} are used to construct symplectic integrators in Section~\ref{Sec:DiscreteMap}. On the other hand, we introduce optimal control problems and how they can be associated with a Morse family that will be useful to run the integrability algorithm~\cite{15Morse}. As a result, using the results in the Appendix, we can construct presymplectic integrators for optimal control problems and provide an example in Section~\ref{Sec:Presympl}.

\section{Hamilton's equations}\label{Sec:HamEq}
As described in~\cite{LiMarle}, the cotangent bundle $T^*Q$ of a  differentiable manifold $Q$ is equipped with a canonical exact symplectic structure $\omega_Q=-d\theta_Q$, where $\theta_Q$ is the canonical 1-form on $T^*Q$. In canonical bundle coordinates $(q^i, p_i)$ on $T^*Q$, 
$
\theta_Q= p_i\, \mathrm{d}q^i$ and
$\omega_Q= \mathrm{d}q^i\wedge \mathrm{d}p_i\; .
$
The Hamiltonian vector field associated with a Hamiltonian function $H: T^*Q\rightarrow {\mathbb R}$ must satisfy:
$
\imath_{X_H}\omega_Q=\mathrm{d}H\; , 
$
whose integral curves are solution to Hamilton's equations: 
\begin{equation*}
	\frac{\mathrm{d}q^i}{\mathrm{d}t}=\frac{\partial H}{\partial p_i}\; ,\quad
	\frac{\mathrm{d}p_i}{\mathrm{d}t}=-\frac{\partial H}{\partial q^i}\; .
\end{equation*}
Some fundamental properties of the Hamiltonian dynamics are:
\begin{itemize}
	\item  Preservation of energy, that is, the Hamiltonian function is preserved:  $$0=\omega_Q(X_H, X_H)=dH(X_H)=X_H(H)\,.$$
	
	\item Preservation of the symplectic form, that is, the Lie derivative of the 2-form $\omega_Q$ with respect to the Hamiltonian vector field vanishes: $L_{X_H}\omega_Q=0$. Equivalently, if $\{\phi^t_{X_H}\}$ is the flow of $X_H$, then
	\[
	(\phi^t_{X_H})^*\omega_Q=\omega_Q\; .
	\]
\end{itemize}
Symplectic integrators~\cite{hairer} were designed to preserve the configuration manifold and preserve the canonical symplectic form.

\section{From retraction maps to discretization maps}  \label{Sec:DiscreteMap}
As described in~\cite{AbMaSeBookRetraction}, a retraction map on a manifold $M$ is a smooth map 
$R\colon U\subseteq TM \rightarrow M$ where $U$ is an open subset containing the zero section  of the tangent bundle 
such that the restriction map $R_x=R_{|T_xM}\colon T_xM\rightarrow M$ satisfies
\begin{enumerate}
	\item $R_x(0_x)=x$ for all $x\in M$,

	\item 
	${\rm D}R_x(0_x)=T_{0_x}R_x={\rm Id}_{T_xM}$ where we identify $T_{0_x}T_xM\simeq T_xM$.
\end{enumerate}

\begin{example}
If  $(M, g)$ is a Riemannian manifold, then the exponential map $\hbox{exp}^g: U\subset TM\rightarrow M$ is a typical example of retraction map: 
$
\hbox{exp}^g_x(v_x)=\gamma_{v_x}(1), 
$
where $\gamma_{v_x}$ is the unique  Riemannian geodesic satisfying $\gamma_{v_x}(0)=x$ and $\gamma'_{v_x}(0)=v_x$~\cite{doCarmo}. 
\end{example}

In \cite{21MBLDMdD}  retraction maps have been used to define discretization maps $R_d\colon U \subset TM  \rightarrow M\times M$, where  $U$ is an open neighbourhood of the zero section of $TM$, 
\begin{eqnarray*}
	R_d\colon U \subset TM & \longrightarrow & M\times M\\
	v_x & \longmapsto & (R^1(v_x),R^2(v_x))\, .
\end{eqnarray*}
Discretization maps satisfy the following properties:
\begin{enumerate}
	\item $R_d(0_x)=(x,x)$, for all $x\in M$.
	\item $T_{0_x}R^2_x-T_{0_x}R^1_x={\rm Id}_{T_xM} \colon T_{0_x}T_xM\simeq T_xM \rightarrow T_xM$ is equal to the identity map on $T_xM$ for any $x$ in $M$.
\end{enumerate}
Thus, the discretization map $R_d$ is a local diffeomorphism.

\begin{example}
Examples of discretization maps on Euclidean vector spaces are:
\begin{itemize}
	\item Explicit Euler method:  $R_d(x,v)=(x,x+v).$
	\item Midpoint rule:  $R_d(x,v)=\left( x-\dfrac{v}{2}, x+\dfrac{v}{2}\right).$
		\item $\theta$-methods with $\theta\in [0,1]$:  \hspace{3mm} $R_d(x,v)=\left( x-\theta \, v, x+ (1-\theta)\, v\right).$
\end{itemize}
\end{example}


\subsection{Cotangent lift of discretization maps}

We want to define a discretization map on $T^*Q$, that is, 
$R^{T^*}_d: TT^*Q \rightarrow T^*Q\times T^*Q$. The domain lives where the Hamiltonian vector field takes value. Such a map will be obtained by cotangently lifting a discretization map $R_d\colon TQ\rightarrow Q\times Q$ so that the construction $R^{T^*}_d$ will be a symplectomorphism. In order to do that, we need the following three symplectomorphisms (see \cite{21MBLDMdD} for more details): 
\begin{itemize}
\item The cotangent lift of a diffeomorphism $F: M_1\rightarrow M_2$ defined by:
	\begin{equation*} 
	\hat{F}: T^*M_1 \longrightarrow  T^*M_2 \mbox{ such that } 
\hat{F}=(TF^{-1})^*.
	\end{equation*}
	\item The canonical symplectomorphism:
	\begin{equation*} \alpha_Q\colon T^*TQ  \longrightarrow  TT^*Q  \mbox{ such that }  \alpha_Q(q,v,p_q,p_v)= (q,p_v,v,p_q).
	\end{equation*}

	\item  The symplectomorphism between $(T^*(Q\times Q), \omega_{Q\times Q})$     and   
	$(T^*Q\times T^*Q, \Omega_{12}=pr_2^*\omega_Q-pr^*_1\omega_Q)$:
		\begin{equation*}
	\Phi:T^*Q\times T^*Q \longrightarrow T^*(Q\times Q)\; , \; 
	\Phi(q_0, p_0; q_1, p_1)=(q_0, q_1, -p_0, p_1).	\end{equation*}
	\end{itemize}
The following diagram summarizes the construction procress from $R_d$ to $R_d^{T^*}$:
	\begin{equation*}
\xymatrix{ {{TT^*Q }} \ar[rr]^{{{R_d^{T^*}}}}\ar[d]_{\alpha_{Q}} && {{T^*Q\times T^*Q }}  \\ T^*TQ \ar[d]^{\pi_{TQ}}\ar[rr]^{	\widehat{R_d}}&& T^*(Q\times Q)\ar[u]_{\Phi^{-1}}\ar[d]^{\pi_{Q\times Q}}\\ TQ \ar[rr]^{R_d} && Q\times Q }
\end{equation*}

\begin{proposition}\cite{21MBLDMdD}
	Let $R_d\colon TQ\rightarrow Q\times Q$ be a discretization map on $Q$. Then $${{R_d^{T^*}=\Phi^{-1}\circ \widehat{R_d}\circ \alpha_Q\colon TT^*Q\rightarrow T^*Q\times T^*Q}}$$ 
is a discretization map  on $T^*Q$.
\end{proposition}
\begin{corollary}\cite{21MBLDMdD} The discretization map
 ${{R_d^{T^*}}}=\Phi^{-1}\circ 	(TR_d^{-1})^* \circ \alpha_Q\colon T(T^*Q)\rightarrow T^*Q\times T^*Q$ is a symplectomorphism between $(T(T^*Q), {\rm d}_T \omega_Q)$ and $(T^*Q\times T^*Q, \Omega_{12})$.
\end{corollary}
\begin{example}\label{example3} On $Q={\mathbb R}^n$ the discretization map 
	$R_d(q,v)=\left(q-\frac{1}{2}v, q+\frac{1}{2}v\right)$ is cotangently lifted to
		$$R_d^{T^*}(q,p,\dot{q},\dot{p})=\left( q-\dfrac{1}{2}\,\dot{q}, p-\dfrac{\dot{p}}{2}; \; q+\dfrac{1}{2}\, \dot{q}, p+\dfrac{\dot{p}}{2}\right)\, .$$
\end{example}

\subsection{Symplectic methods for Hamilton's equations}

For Hamilton's equation we automatically produce a symplectic integrator using 
the discretization map $R_d^{T^*}\colon TT^*Q \rightarrow T^*Q\times T^*Q$ which is the cotangent lift of a discretization map on $Q$.

\begin{proposition}\cite{21MBLDMdD}
The numerical method defined by 
	$$h\, X_H \left(\tau_{TQ}\left( \left(R^{TT^*Q}_d\right)^{-1}(q_k, p_k; q_{k+1}, p_{k+1})\right)\right)=\left(R^{TT^*Q}_d\right)^{-1}(q_k, p_k; q_{k+1}, p_{k+1})\,, $$
	is a symplectic integrator for the Hamiltonian system given by $H\colon T^*Q\rightarrow \mathbb{R}$.
	\end{proposition}
\section{Optimal control problems and Morse families}\label{Sec:OptMorse}
An optimal control problem (OCP) is given by a vector field depending on parameters called controls, a cost function and some end-point conditions. A solution of an OCP must be an integral curve of the vector field for specific controls, $\dot{q}=X(q,u)$, so that the functional $\int^{t_f}_{t_0} F(q(t),u(t))\, {\rm d }t,$ is minimized and the end-point conditions satisfied. 

Typically, OCP are solved using Pontryagin's Maximum Principle~\cite{1997Jurdjevic,Pontryagin} that provides the problem with a Hamiltonian framework. 
Let $U$ be the set of admissible controls, the associated Pontryagin's Hamiltonian function is:
\begin{equation*}
H \colon T^*Q\times U \rightarrow \mathbb{R}\, ,\quad 
H(q,p,u)=\langle p, X(q,u) \rangle - F(q,u)\, ,\end{equation*}
where $\langle\cdot, \cdot \rangle$ denotes the natural pairing between $T_q^*Q$ and $T_qQ$.
A Morse family is another geometric object that can be used to define Lagrangian submanifolds. Such a notion was first  introduced by L. H\"ormander~\cite{Hormander}.  It is proved in~\cite{15Morse} that the Pontryagin's Hamiltonian function could be a Morse family over the projection ${\rm pr}_1\colon T^*Q\times U \rightarrow T^*Q$ onto the first factor if the image of the differential of $H$ and the conormal bundle
\begin{equation*}(\ker \, {\rm T}\pi)^0=\left\{\alpha\in T_\mu ^*(T^*Q\times U)\; |\; \langle \alpha, v\rangle=0, \hbox{ for all } v\in \ker T_{\mu}\pi\right\}\subset T^*(T^*Q\times U)
\end{equation*}
are transverse in $T^*(T^*Q\times U)$, that is, 
\begin{equation*} {\rm T}_\alpha( {\rm d}H(T^*Q\times U))+T_\alpha (\ker \, {\rm T}\pi)^0=T_\alpha (T^*(T^*Q\times U)),
\end{equation*}
for all $\alpha\in (\ker \, {\rm T}\pi)^0 \cap  {\rm d}H(T^*Q\times U)  \subseteq T^*(T^*Q\times U)$.

\begin{proposition}\cite{15Morse} Pontryagin's Hamiltonian   $H\colon T^*Q\times U  \rightarrow \mathbb{R}$ defines a \textbf{Morse family over} the projection ${\rm pr}_1\colon T^*Q\times U \rightarrow T^*Q$ onto the first factor if and only if  the matrix
	\begin{equation*}
	{\rm D}_{(q,p,u)} \begin{pmatrix} \dfrac{\partial H}{\partial u}\end{pmatrix}=\begin{pmatrix}\dfrac{\partial^2 H}{\partial q^i \partial u^a} & \dfrac{\partial^2 H}{\partial p_i \partial u^a} & \dfrac{\partial^2 H}{\partial u^a \partial u^b}
	\end{pmatrix}_{(q,p,u)} 
	\end{equation*}
	has maximum rank for all $(q,p,u)\in T^*Q\times U$ such that $\dfrac{\partial H}{\partial u}(q,p,u)=0$.
\end{proposition}

When the controls are in the interior of the set $U$, the necessary conditions of Pontryagin's Maximum Principle can be rewritten as the following Lagrangian submanifold ${\mathcal L}_H$ of $(T^*T^*Q, \omega_{T^*Q})$:
\[
{\mathcal L}_H= \left\{(q, p, P_q, P_p)\;\left| \exists \; u\in U \; \mbox{s. t. } \begin{array}{l}    P_q=\frac{\partial H}{\partial q}(q,p,u), P_p=\frac{\partial H}{\partial p}(q,p,u), \\  \frac{\partial H}{\partial u}(q,p,u)=0\,  \end{array}\right.\right\}\, .
\]

The OCP is regular if $(\pi_{T^*Q})_{|{\mathcal L}_H}: {\mathcal L}_H\rightarrow T^*Q$ is a local diffeomorphism, otherwise it is called singular.
Observe that in general ${\mathcal L}_H$ is not horizontal, that is, it is not transverse to the fibers of the canonical cotangent projection $\pi_{T^*Q}$. Consequently, it is not the image of the differential of a function on $T^*Q$ \cite{15Morse}.

Due to the symplectomorphism between $(TT^*Q, d_T\omega_Q)$ and $(T^*T^*Q, \omega_{T^*Q})$ described in~\cite{TuHamilton}, the dynamics of an optimal control problem can also be given as the following Lagrangian submanifold in $(TT^*Q, d_T\omega_Q)$:
\begin{equation*}
S_{H}=\sharp_{\omega_{T^*Q}}({\mathcal L}_H)=\{v\in TT^*Q \, | \; i_v\omega_{T^*Q}\in {\mathcal L}_H\} 
\end{equation*}
Thus, a solution of the OCP  is a curve $\sigma$ in $T^*Q$ such that $\dot{\sigma}(t)$ lies in $S_{H}$.

In general, the solutions of the OCP are consistently defined in a submanifold of $T^*Q$ contained in $\tau_{T^*Q}(S_H)\subseteq T^*Q$. Thus, the integrability algorithm~\cite{1995MMT} can be used to obtain the integrable part of $S_{H}$ in $T^*Q$. First, we define
$S^0_H=S_{H}$ and $P_H^0=\tau_{T^*Q}(S_H)$.

The following steps of the algorithm are defined by 
\begin{equation*}
S^k_H=T P_H^{k-1}\cap S_H^{k-1}, \quad P_H^k=\tau_{T^*Q}(S_H^k).
\end{equation*}
If the algorithm stabilizes at step $k_f$ of the constraint algorithm, there exists a final submanifold (possibly empty or singular) satisfying $S^{k_f}_H=T P^{k_f}_H \cap S^{k_f}_H$, that will be denoted by $S^f_{H}$. 
On the base manifold $T^*Q$ and the tangent bundle $TT^*Q$ the algorithm generates the following two sequences of submanifolds in $T^*Q$ and $TT^*Q$, respectively:
\begin{equation*}
\begin{array}{c}
P^f_H \stackrel{i_{k_f}}{\hooklongrightarrow} P^{k_f-1}_H \stackrel{i_{k_f-1}}{\hooklongrightarrow}\cdots  P^1_H \stackrel{i_{1}}{\hooklongrightarrow} P^0_H \stackrel{i_{0}}{\hooklongrightarrow} T^*Q\, ,
\\
S^f_H \stackrel{j_{k_f}}{\hooklongrightarrow} S^{k_f-1}_H \stackrel{j_{k_f-1}}{\hooklongrightarrow}\cdots S^1_H \stackrel{j_{1}}{\hooklongrightarrow} S^0_H \stackrel{j_{0}}{\hooklongrightarrow} TT^*Q\, .
\end{array}
\end{equation*}

As a consequence, for every $\alpha$ in $P^f_H$ there exists $V$ in $T_{\alpha}S^f_H\subset TT^*Q$. 
Hence, the original dynamical system has solution in the submanifold $P^f_H$. 
Denote by $i_f: P^f_H\hookrightarrow T^*Q$ the canonical inclusion and by $\omega_f=i_f^*\omega_{T^*Q}$ the pullback of the canonical symplectic 2-form on $T^*Q$. Note that $\omega_f$ is now a presymplectic 2-form (see Appendix).

In conclusion, a solution to the $OCP$  is a curve $\sigma$ on $P^f_H$ such that there exist controls $u$ satisfying  
\begin{equation*}\label{presymplectic-equation}
i_{\dot{\sigma}(t)}\omega_f (\sigma(t))=dH_{f}^{u}(\dot{\sigma}(t))\, ,
\end{equation*}
 where $H_f^u: S^f_H\rightarrow {\mathbb R}$ is given by $H_f^u(\dot{\sigma}(t))=H(j_f(\dot{\sigma}(t)),u(t))$. Therefore, the dynamics that we need to preserve with our numerical methods is presymplectic instead of the most classical symplectic preservation property.

 \begin{proposition}\label{proppre} The submanifold
 $S^f_H$ is a Lagrangian submanifold of the presymplectic manifold  $(TP^f_H, d_T\omega_f)$.
 \end{proposition}
 \begin{proof}
 The results follows because $\omega_f=i_f^*\omega_{T^*Q}$, ${S}_H$ is a Lagrangian submanifold of $(TT^*Q, d_T\omega_Q)$ and 
 $
 d_T\omega_f=d_Ti_f^*\omega_{T^*Q}=(Ti_f)^*d_T\omega_{T^*Q}
 $.
 \end{proof}

\section{Presymplectic integrators for optimal control problems}\label{Sec:Presympl}
In this section we will use the cotangent lift of a discretization map to define a presymplectic integrator for optimal control problems, once we have run the integrability algorithm and know the final submanifold $P^f_H$ of $T^*Q$. 

We restrict the cotangent lift of a discretization map $R^{T^*}_d\colon T(T^*Q)\rightarrow T^*Q\times T^*Q$ to the submanifold $TP^f_H$ and define the submanifold $P^f_{H, d}$  of $T^*Q\times T^*Q$ by
\[
P^f_{H, d}\colon =R_d^{T^*}(TP^f_H)\,.
\]
Introducing the inclusion $j_f^d:P^f_{H, d}\hookrightarrow T^*Q\times T^*Q$ the following diagram summarizes the construction process:
\begin{equation*}
\xymatrix{ {{T^*Q\times T^*Q }} && {TT^*Q }\ar[ll]_{{{R_d^{T^*}}}} \ar[rd]^{\tau_{T^*Q}}&&
	T^*T^*Q\ar[ll]_{\sharp_{\omega_Q}}\ar[ld]^{\pi_{T^*Q}}&\ar@{_(->}[l] {\mathcal  L}_H
	\\ 
	P^f_{H, d}\ar[u]_{j^d_f}&&TP^f_H\ar[u]_{j_f}\ar[rd]_{\tau_{P^f_H}} \ar[ll]_{{{R_d^{T^*}|_{TP^f_H}}}}&T^*Q&\\
	&&&P^f_H\ar[u]_{i_f} &
}
\end{equation*}

\begin{definition}\label{method}
We define the OCP geometric integrator as
\[
\left\{\begin{array}{l}
\frac{1}{h}\left(R_d^{T^*}\right)^{-1}(q_k, p_k; q_{k+1}, p_{k+1})\in  \left(S^f_H\right)_{\tau_{T^*Q}\left(\left(R_d^{T^*}\right)^{-1}(q_k, p_k; q_{k+1}, p_{k+1})\right)}\, ,\\
 (q_k, p_k; q_{k+1}, p_{k+1})\in P^f_{H, d}
\end{array}\right.
\]
\end{definition}
Propositions \ref{proppre}  and \ref{propopre} guarantees the presymplecticity of the method.
\begin{proposition}
The OCP geometric integrator in Definition~\ref{method} preserves the presymplectic 2-form 
$\Omega_{f, d}=(j_f^d)^*\Omega_{12}$.
\end{proposition}

\subsection{Example}
As an academic example we consider the singular optimal control problem on $\mathbb{R}^2$ given by the control equations 
$
\dot{x}=f(x)+u_1, \; \dot{y}=y\,
$,
where $u_1$, $u_2 \in \mathbb{R}$,
and the cost functional
\[
\int_{t_0}^{t_f} \left( 
\frac{1}{2}x^2+\frac{1}{2}y^2+xu_1+yu_2+\frac{1}{2}u_1^2\right)\, dt\; .
\]
Then Pontryagin's Hamiltonian is
\[
H(x,y,p_x,p_y,u_1,u_2)=p_x (f(x)+u_1)+p_y y-\frac{1}{2}x^2-\frac{1}{2}y^2-xu_1-yu_2-\frac{1}{2}u_1^2\, .
\]
The Lagrangian submanifold ${\mathcal L}_H$ of $(T^*(T^*\mathbb{R}^2),\omega_{T^*\mathbb{R}^2})$ is
\begin{eqnarray*}
{\mathcal L}_H&=&
\{(x,y, p_x, p_y, P_{x}, P_y, P_{p_x}, P_{p_y})\ |\; 
P_{x}=p_x\, f'(x)-x-u_1, P_y=p_y-y-u_2,\\
&&P_{p_x}=f(x)+u_1, P_{p_y}=y, \; p_x-x-u_1=0, \; y=0\}\, .
\end{eqnarray*}
Therefore, we obtain the following Lagrangian submanifold of $(TT^*\mathbb{R}^2,{\rm d}_T\omega_{\mathbb{R}^2})$:
\begin{eqnarray*}
{S}_H&=&
\{(x,y, p_x, p_y, \dot{x}, \dot{y}, \dot{p}_x, \dot{p}_y)\ |\; 
\dot{x}=f(x)+u_1, \dot{y}=y,\dot{p}_x=-p_x\, f'(x)+x+u_1,\\
&& \dot{p}_y=-p_y+y+u_2, \; p_x-x-u_1=0, \; y=0\}\,.
\end{eqnarray*}
Applying the constraint algorithm we immediately deduce that 
\[
P^f_H=\{(x,0, p_x, p_y)\in {\mathbb R}^4\}\equiv {\mathbb R}^3
\]
which is a presymplectic manifold with the 2-form $\omega_f=dx\wedge dp_x$
whose $\ker \omega_f=\hbox{span}\{ {\partial}/{\partial p_y}\}$. The corresponding Lagrangian submanifold of the presymplectic manifold 
$(TP^f_H, d_T\omega_f)$
is:
\begin{eqnarray*}
S_H^f&=&\{(x, p_x, p_y, \dot{x}, \dot{p}_x, \dot{p}_y)\ |\; 
\dot{x}=f(x)+u_1, \dot{p}_x=-p_x f'(x)+x+u_1,\\
&& \dot{p}_y=-p_y+u_2, \; p_x-x-u_1=0 \}\, .
\end{eqnarray*}
Applying now the discretization map derived in Example \ref{example3} we obtain that 
\[
P^f_{H,d}=\{(x_k, 0, (p_x)_k, (p_y)_k, x_{k+1}, 0, (p_x)_{k+1}, (p_y)_{k+1})\in {\mathbb R}^8\}\equiv {\mathbb R}^3\times {\mathbb R}^3
\]
and the resulting presymplectic integrator is described by the equations: 
\begin{eqnarray*}
\frac{x_{k+1}-x_k}{h}&=&f\left(\frac{x_k+x_{k+1}}{2}\right)+(u_1)_k,\\ \frac{(p_x)_{k+1}-(p_x)_k}{h}&=&-\frac{(p_x)_{k}+(p_x)_{k+1}}{2} f'\left(\frac{x_k+x_{k+1}}{2}\right)+
\frac{x_k+x_{k+1}}{2}+(u_1)_k,\\
 \frac{(p_y)_{k+1}-(p_y)_k}{h}&=&-\frac{(p_y)_{k}+(p_y)_{k+1}}{2}+(u_2)_k, \\
  0&=&\frac{(p_x)_{k}+(p_x)_{k+1}}{2}-\frac{x_k+x_{k+1}}{2}-(u_1)_k\, .
\end{eqnarray*}
The discrete equations for $p_y$ and $u_2$ are decoupled from the rest that can be written in a more compact way for $(x_k,(p_x)_k,x_{k+1},(p_x)_{k+1})$:
\begin{eqnarray*}
\frac{x_{k+1}-x_k}{h}&=&f\left(\frac{x_k+x_{k+1}}{2}\right)+\frac{(p_x)_{k}+(p_x)_{k+1}}{2}-\frac{x_k+x_{k+1}}{2},\\ \frac{(p_x)_{k+1}-(p_x)_k}{h}&=&-\frac{(p_x)_{k}+(p_x)_{k+1}}{2}f'\left(\frac{x_k+x_{k+1}}{2}\right)+
\frac{x_k+x_{k+1}}{2}\\
&&+\frac{(p_x)_{k}+(p_x)_{k+1}}{2}-\frac{x_k+x_{k+1}}{2}\,.
\end{eqnarray*}

\section{Conclusions and future
work}

In this paper we have only studied normal solutions, both regular and singular, for optimal control problems. The same technique can be applied for abnormal solutions where Pontryagin's Hamiltonian does not depend on the cost function. Moreover, we plan to construct geometric integrators for Dirac systems~\cite{Melvin}.

\section*{Appendix: Presymplectic geometry and Lagrangian submanifolds}

As introduced in~\cite{GuzMarr}, a preymplectic structure on a finite dimensional manifold $M$ is a closed 2-form $\omega$ on $M$. We say that $(M, \omega)$ is a presymplectic manifold. The kernel of the presymplectic structure at a point $x$ in $M$ is a vector subspace of the tangent space of $M$ at $x$ that it is  not necessarily zero as in the symplectic case. Remember that 
$
\ker \omega_x=\{v\in T_x M\; |\; i_v \omega_x=0\}
$.
If $\dim \ker\omega_x=0$ for all $x\in M$, the presymplectic structure is non degenerate. Hence, it is a symplectic structure~\cite{90GuiStern,LiMarle}. 

\begin{definition}
A submanifold ${\mathcal L}$ of dimension $r$ of the presymplectic manifold $(M, \omega)$ with canonical inclusion $i: {\mathcal L}\hookrightarrow M$ is said to be Lagrangian if the pullback of $\omega$ by the inclusion vanishes, that is, $i^*\omega=0$, and
\begin{equation}\label{eq:dim_preL}
r=\frac{ \dim M- \dim \left(\ker \omega_x\right)}{2} +\dim (T_x{\mathcal L}\cap \ker \omega_x), \quad \hbox{for all} \;  x\in M\,.
\end{equation}
\end{definition}
When $\omega$ is a symplectic structure, Equation~\eqref{eq:dim_preL} implies that the dimension of ${\mathcal L}$ is half of the dimension of $M$ and we recover the classical definition of Lagrangian submanifold in symplectic geometry~\cite{LiMarle,Weinstein}.

A smooth map $f: M\rightarrow N$ between two presymplectic manifolds $(M, \omega_M)$ and $(N, \omega_N)$ is a presymplectic map if $f$ preserves the presymplectic structures, that is,  $f^*\omega_N=\omega_M$. From that notion, it is possible to construct the following Lagrangian submanifolds.  The proofs come from the above definitions.
\begin{proposition} \label{Prop:preGraphf} Let $(M, \omega_M)$ and $(N, \omega_N)$ be presymplectic manifolds.
If $f: M\rightarrow N$  is a presymplectic diffeomorphism, then 
\[
\hbox{\rm Graph} f=\{(x, f(x)) \, | \, x\in M\}
\]
is a $m$-dimensional  Lagrangian submanifold of $(M\times N, \Omega_{M\times N}=\omega_M-\omega_N)$, where $m$ is the dimension of $M$.
\end{proposition}

\begin{proposition}\label{propopre}
Let $(N, \omega_N)$ be a presymplectic manifold and $f: M\rightarrow N$ a diffeomorphism. Then $(M,\omega_M)$ is a presymplectic manifold with the presymplectic structure $\omega_M= f^*\omega_N$. Moreover, if ${\mathcal L}_N$  is a Lagrangian submanifold of $(N, \omega_N)$, then ${\mathcal L}_M=f^{-1}({\mathcal L}_N)$ is a Lagrangian submanifold of $(M,\omega_M)$.  
\end{proposition}


%
%

\end{document}